\documentclass[a4paper,12pt]{amsart}
\usepackage{amscd,amssymb}

\usepackage[all]{xy}
\renewcommand{\mod}{\operatorname{mod}\nolimits}

\newcommand{\add}{\operatorname{add}\nolimits}

\newcommand{\Hom}{\operatorname{Hom}\nolimits}
\newcommand{\End}{\operatorname{End}\nolimits}

\renewcommand{\Im}{\operatorname{Im}\nolimits}
\newcommand{\Ker}{\operatorname{Ker}\nolimits}
\newcommand{\Coker}{\operatorname{Coker}\nolimits}
\newcommand{\rrad}{\mathfrak{r}}

\newcommand{\Soc}{\operatorname{Soc}\nolimits}

\newcommand{\gldim}{\operatorname{gldim}\nolimits}

\newcommand{\DTr}{\operatorname{DTr}\nolimits}
\newcommand{\TrD}{\operatorname{TrD}\nolimits}
\newcommand{\Ext}{\operatorname{Ext}\nolimits}

\newcommand{\op}{{\operatorname{op}\nolimits}}

\newcommand{\id}{{\operatorname{id}\nolimits}}
\newcommand{\pd}{{\operatorname{pd}\nolimits}}

\renewcommand{\L}{\Lambda}

\newcommand{\I}{{\mathcal I}}

\renewcommand{\P}{{\mathcal P}}

\newcommand{\extto}{\xrightarrow}

\newtheorem{lem}{Lemma}[section]
\newtheorem{prop}[lem]{Proposition}
\newtheorem{cor}[lem]{Corollary}
\newtheorem{thm}[lem]{Theorem}
\theoremstyle{definition}
\newtheorem{defin}[lem]{Definition}
\newtheorem*{remark}{Remark}

\newtheorem{example}[lem]{Example}
\newtheorem{conj}[lem]{Conjecture}

\begin{document}

\title{Relative homology and maximal $l$-orthogonal modules}
\author[Lada]{Magdalini Lada}
\thanks{Thanks .. for support.}
\address{Magdalini Lada\\Institutt for matematiske fag\\
NTNU\\ N--7491 Trondheim\\ Norway} \email{magdalin@math.ntnu.no}

\date{\today}
\maketitle

\section*{Introduction}

Maximal $l$-orthogonal modules for artin algebras were introduced by Iyama in ~\cite{I}
and ~\cite{I1}, and were used to define a natural setting for developing a `higher
dimensional' Auslander-Reiten theory. In the second of these papers Iyama conjectures
that the endomorphism ring of any two maximal $l$-orthogonal modules, $M_1$ and $M_2$,
are derived equivalent. He proves the conjecture for $l=1$, and for $l>1$ he gives some
orthogonality condition on $M_1$ and $M_2$, such that the
$\End_\L(M_2)^\op$-$\End_\L(M_1)$-bimodule $\Hom_\L(M_2,M_1)$ is tilting, which implies
that the rings $\End_\L(M_2)$ and $\End_\L(M_1)$ are derived equivalent (see ~\cite{H}).
The purpose of this paper is to characterize tilting modules of the form
$\Hom_\L(M_2,M_1)$ in terms of the relative theories induced by the $\L$-modules $M_1$
and $M_2$, thus getting a generilization of Iyama's result. Relative homological algebra,
which we use throughout this paper, was developed by M. Auslander and \O. Solberg in a
series of three papers ~\cite{AS1}, ~\cite{AS2}, ~\cite{AS3} and was used to study the
representation theory of artin algebras.

Iyama's conjecture for maximal $l$-orthogonal modules is motivated by the connection
between maximal $l$-orthogonal modules and non-commutative crepant resolutions, which was
shown in ~\cite{I1}. In particular Iyama proved that if $R$ is a complete regular local
ring of dimension $d\geq 2$ and $\L$ is an $R$-order which is not an isolated
singularity, then a Cohen-Macaulay $\L$-module $M$, which is a generator and cogenerator
for $\mod\L$, gives a non-commutative crepant resolution if and only if $M$ is maximal
$(d-2)$-orthogonal. Considering maximal $l$-orthogonal modules as analogs of modules
giving crepant resolutions, Iyama's conjecture for maximal $l$-orthogonal modules is the
analog of the Bondal-Orlov-Van den Bergh conjecture for crepant resolutions (see
~\cite{BO}, ~\cite{V}).

Let us fix some notation that we will use in the rest of this paper. By $\L$ we will
denote an artin algebra and by $\mod\L$ the category of finitely generated left
$\L$-modules. If $M$ is in $\mod \L$, we denote by $\add M$ the full subcategory of
$\mod\L$, consisting of summands of direct sums of $M$. A \emph{generator-cogenerator}
for $\mod\L$, is a $\L$-module $M$ such that $\add M$ contains all the indecomposable
projective and the indecomposable injective $\L$-modules. Maximal $l$-orthogonal modules
have this property.

In the first section,we recall some concepts from relative homological algebra and we
study further the relative theories induced by generator-cogenerators for $\mod \L$. In
the second section we give Iyama's definition for maximal $l$-orthogonal modules and
state his conjecture. We prove our main theorem which is, as we already mentioned, a
characterization of tilting modules of the form $\Hom_\L(M_2,M_1)$ and we apply it to
maximal $l$-orthogonal modules. In this way we are able to give a condition on two
maximal $l$-orthogonal modules, so that the conjecture is true. We prove that Iyama's
orthogonality condition implies our condition and give an example which shows that our
condition is actually weaker.

\section{Relative and absolute homology}
The aim of this section is first to recall some basic definitions and results on relative
homology, and second to look especially at the relative theory induced by a
generator-cogenerator for $\mod\L$. In particular we relate the global dimension of the
endomorphism ring of a generator-cogenerator  $M$ in $\mod\L$, with the relative - with
respect to $M$ - global dimension of $\L$ (it will be defined precisely later). Moreover
we will compare the relative homology induced by such a module, with the ordinary
absolute homology. For unexplained terminology and results, we refer to ~\cite{AS1} and
~\cite{AS2}.

Let $F$ be an additive sub-bifunctor of $\Ext_\L^1(-,-)\colon (\mod\L)^\op \times
\mod\L\to \mathrm{Ab}$, where $\mathrm{Ab}$ denotes the category of abelian groups. A
short exact sequence $ (\eta)\colon 0\to A \to B \to C \to 0$, in $\mod\L$, is called
\emph{$F$-exact} if $\eta$ is in $F(C,A)$. A $\L$-module $P$ is called
\emph{$F$-projective} if for any $F$-exact sequence $0\to A \to B \to C \to 0$, the
sequence $\Hom_\L(P,B)\to \Hom_\L(P,C)\to 0$ is exact. Dually, a $\L$-module $I$ is
called \emph{$F$-injective} if for any $F$-exact sequence $0\to A \to B \to C \to 0$, the
sequence  $\Hom_\L(B,I)\to \Hom_\L(A,I)\to 0$ is exact. We denote by $\P(F)$ the full
subcategory of $\mod\L$ consisting of all $F$-projective $\L$-modules and by $\I(F)$ the
one consisting of all $F$-injective $\L$-modules. Note that if $\P(\L)$ and $\I(\L)$
denote the subcategories of projective and injective modules in $\mod\L$ respectively, we
have that $\P(\L)\subseteq\P(F)$ and $\I(\L)\subseteq\I(F)$.

In this paper we work with sub-bifunctors of $\Ext_\L^1(-,-)$ of the following special
form. Let $M$ be in $\mod \L$. For each pair of $\L$-modules $A$ and $C$ we define

\begin{multline}F_M(C,A)=\{0\to A\to B \to C\to 0 \mid \Hom_\L(M,B)\to \notag\\\Hom_\L(M,C)\to 0 \mathrm{\
is \ exact}\}
\end{multline}
and dually
\begin{multline}F^M(C,A)=\{0\to A\to B \to C\to 0 \mid \Hom_\L(B,M)\to \notag\\
\Hom_\L(A,M)\to 0 \mathrm{\ is \ exact}\}.
\end{multline}

It was shown in ~\cite{AS1}, that the above assignments give additive sub-bifunctors of
$\Ext_\L^1(-,-)$. It is straightforward that $\P(F_M)=\P(\L)\cup \add M$ and
$\I(F^{M})=\I(\L)\cup \add M$. Moreover, it is known that for any short exact sequence
$0\to A\to B \to C\to 0$ and any $\L$-module $M$, the sequence $\Hom_\L(M,B)\to
\Hom_\L(M,C)\to 0$ is exact if and only if the sequence $\Hom_\L(B,\DTr M)\to
\Hom_\L(A,\DTr M)\to 0$ is exact. Using this fact it is easy to see that $F_M=F^{\DTr
M}$, so $\I(F_M)=\I(\L)\cup \add \DTr M$, and $F^M=F_{\TrD M}$, so $\P(F^M)=\P(\L)\cup
\add \TrD M$. Hence we have a complete picture of the $F$-projectives and $F$-injectives
for sub-bifunctors of the above form.

Another nice property of a sub-bifunctor $F=F_M$ of $\Ext_\L^1(-,-)$ is that it \emph{has
enough projectives}, in the sense that for any $\L$-module $C$, there exists an $F$-exact
sequence $0\to K \to P\to C\to 0$, with $P$ in $\P(F)$. Note that the map $P\to C$ is
nothing but a right $\P(F)$-approximation of $C$ which we know is an epimorphism since
$\P(\L)\subseteq \P(F)$. Dually we see that $F=F_M$ \emph{has enough injectives}.

Let $F=F_M$ and $C$ in $\mod\L$. An \emph{$F$-projective resolution} of $C$ is an exact
sequence
\[ \cdots\to P_l\extto{f_l} P_{l-1}\to \cdots \to P_1\extto{f_1} P_0 \extto{f_0} C \to 0\]
where $P_i$ is in $\P(F)$ and each short exact sequence $0\to\Im f_{i+1}\to P_i \to \Im
f_i\to 0$ is $F$-exact. Note that such a sequence exists for any $\L$-module since
$F=F_M$ has enough projectives. The sequence is called a \emph{minimal $F$-projective
resolution} if in addition each $P_i \to \Im f_i$ is a minimal map. The
\emph{F-projective dimension} of $C$, which we denote by $\pd_F C$, is defined to be the
smallest $n$ such that there exists an $F$-projective resolution
\[0\to P_n\to\cdots\to P_1\to P_0 \to C\to 0.\] If such $n$ does not exist we set
$\pd_F C=\infty$. Dually, we can define the notion of a \emph{(minimal) F-injective
resolution} and the \emph{F-injective dimension}, $\id_F C$, of $C$. Then, the
\emph{F-global dimension} of $\L$ is defined as:
\[\gldim_F \L =\sup\{\pd_F C  \mid C\in\mod\L\}\left(=\sup\{\id_F C \mid C \in\mod\L\}\right).\]

In the special case where $M$ is a generator-cogenerator for $\mod\L$, we have the
following nice connection between the global dimension of the endomorphism ring of $M$
and the relative global dimension of $\L$.

\begin{prop}\label{gldim}
Let $M$ be a generator-cogenerator for $\mod\L$. Then, for any positive integer $l$, the
following are equivalent:
\begin{enumerate}
\item[(a)] $\gldim\End_{\L}(M)\leq l+2$
\item[(b)] $\gldim_{F_{M}}\L\leq l$
\item[(c)] $\gldim_{F^{M}}\L\leq l$.
\end{enumerate}
\end{prop}

\begin{proof}
(a)$\Rightarrow$(b) Let $N$ be a $\L$-module and let
\[\cdots\to M_l\to M_{l-1} \extto {f_{l-1}} \cdots\to M_1\to M_0\to N\to 0\] be a
minimal $F_M$-projective resolution of $N$. Set $K_{l-1}=\Ker f_{l-1}$.  We will show
that $K_{l-1}$ is in $\add M$. Applying the functor $\Hom_\L(M,-)$ to the above sequence,
we get the long exact sequence
\begin{multline} \eta_{1} : \ldots\to \Hom_\L(M,M_{l})\to \Hom_\L(M,M_{l-1})
\extto{\Hom(M,f_{l-1})} \cdots\to \notag \\\Hom_\L(M,M_1)\to \Hom_\L(M,M_0)\to
\Hom_\L(M,N)\to 0.
\end{multline}
Note that $\Ker{\Hom(M,f_{l-1})}= \Hom(M,K_{l-1})$. Let also \[0\to N\to
M^0\extto{f}M^1\to\cdots\]  be the beginning of an $F^M$-injective resolution of $N$.
Applying the functor $\Hom_\L(M,-)$ to this sequence, we get the long exact sequence
\[\eta_{2} : 0\to\Hom_\L(M,N)\to\Hom_\L(M,M^0)\extto{\Hom_\L(M,f)}\Hom_\L(M,M^1)\to X\to, 0\]
where the $\End_\L(M)^\op$-module $X$ is the cokernel of the map $\Hom_\L(M,f)$.
\\Combining now the sequences $\eta_1$and $\eta_2$ we get the long exact sequence
\begin{multline}
 \ldots\to \Hom_\L(M,M_{l})\to \Hom_\L(M,M_{l-1}) \to \cdots\to
\Hom_\L(M,M_1)\to \notag \\ \Hom_\L(M,M_0)\to \Hom_\L(M,M^0)\to\Hom_\L(M,M^1)\to X\to 0
\end{multline}
Since the $\L$-modules $M_i$ for $i=0,1,\ldots$ and $M^j$ for $j=0,1$ are in $\add M$,
the $\End_{\L}(M)^\op$-modules $\Hom_\L(M,M_i),i=0,1,\ldots$ and $\Hom_\L(M,M^j),j=0,1$
are projective. So the $\End_{\L}(M)^\op$-module $\Hom_\L(M,K_{l-1})$ is an $(l+2)$-th
syzygy of the $\End_{\L}(M)^\op$-module $X$. Since $\gldim\End_{\L}(M)\leq l+2$, the
module $\Hom_\L(M,K_{l-1})$ has to be a projective $\End_{\L}(M)^\op$-module, so
$K_{l-1}$ is in
$\add M$. Thus $\gldim_{F_{M}}\L\leq l$.\\
(b)$\Rightarrow$(a) Let $X$ be an $\End_{\L}(M)^{\op}$-module and let
\begin{multline}
 \ldots\to \Hom_\L(M,M_{l+2})\extto{d_{l+2}} \Hom_\L(M,M_{l+1}) \to \cdots\to
\Hom_\L(M,M_2)\extto{d_2} \notag \\ \Hom_\L(M,M_1)\extto{d_1} \Hom_\L(M,M_0)\to X\to 0
\end{multline}
be a projective resolution of $X$. Since $M$ is a generator for $\mod\L$, the functor
$\Hom_\L(M,-)$ is full and faithfull hence we have that $\Ker d_i=\Hom_\L(M,K_i)$ where
$K_i$ is the kernel of the morphism $f_i\colon M_i\to M_{i-1}$, with $\Hom_\L(M,f)=d_i$,
for $i>0$. We will show that $\Ker d_{l+1}$ is a projective $\End_\L(M)^{\op}$-module.
The sequence
\[\cdots\to M_{l+2}\extto{f_{l+2}} M_{l+1}\to\cdots\to M_2\to K_1\to 0\]
is an $F_M$-projective resolution of $K_1$. Since $\gldim_{F_{M}}\L\leq l$, we have that
$K_{l+1}$ is in $\add M$ and hence $\Hom_{\L}(M,K_{l+1})$ is a projective
$\End_\L(M)^{\op}$-module. So $\gldim\End_{\L}(M)\leq l+2$.

Thus we have proved the equivalence of $(a)$ and $(b)$. The proof of the equivalence of
$(a)$ and $(c)$ is symmetric.
\end{proof}

Note that the above result was basically already known and can be found, for example in
~\cite{EHIS}, in a different form. But, besides the fact that using relative homology
helps in having a simpler statement, there is also the advantage of getting extra
information about coresolutions of $\L$-modules in $\add\DTr M$ and resolutions of
$\L$-modules in $\add\TrD M$, since $F_M=F^{\DTr M}$ and $F^M=F_{\TrD M}$. In particular,
for selfinjective algebras, we have the following consequence.

\begin{cor}\label{selfinjective}
Let $\L$ be a selfinjective artin algebra and $X$ be in $\mod \L$. Then, for any positive
integer $l$, the following are equivalent:
\begin{enumerate}
\item[(a)] $\gldim\End_{\L}(\L\oplus X)\leq l+2$,
\item[(b)] $\gldim\End_{\L}(\L\oplus \DTr X)\leq l+2$,
\item[(c)] $\gldim\End_{\L}(\L\oplus \TrD X)\leq l+2$.
\end{enumerate}
\end{cor}

\begin{remark}A straightforward consequence of the above corollary is that if $\L$ is a
selfinjective artin algebra, then $\L\oplus X$ is an Auslander generator (that is, a
generator-cogenerator for $\mod \L$ such that the global dimension of its endomorphism
ring gives the representation dimension of $\L$) if and only if $\L\oplus\DTr X$ is an
Auslander generator.
\end{remark}
Let $A$ and $C$ be in $\mod\L$. Knowing that for sub-bifunctors $F=F_M$ of
$\Ext_\L^1(-,-)$ any $\L$-module has an $F$-projective and an $F$-injective resolution,
one can define the right derived functors $\Ext_F^i(C,-)$ and $\Ext_F^i(-,A)$ of
$\Hom_\L(C,-)$ and $\Hom_\L(-,A)$ respectively, in the same way as in the case
$\P(F)=\P(\L)$. Moreover, it can be proved that $\Ext_F^i(C,-)(A)$ is then isomorphic to
$\Ext_F^i(-,A)(C)$ and that for $i=1$ we have that $\Ext_F^1(C,A)=F(C,A)$.

Although $\Ext_F^1(C,A)$ is a subgroup of $\Ext_\L^1(C,A)$, very little is known about
how the $\Ext_{\L}^i(C,A)$-groups and the relative $\Ext_{F}^i(C,A)$-groups are related
for $i>1$. In the next proposition we consider a case where these two coincide. This will
help us in the next section to compare Iyama's condition on maximal orthogonal modules to
our result. For $X$ and $Y$ in $\mod\L$ we write ${X\perp_k Y}$ if $\Ext_\L^i(X,Y)=(0)$,
for $0<i\leq k$. Abusing the notation, we will write ${X\perp_k Y}$ even for $k=0$ but
this will mean no condition on $X$ and $Y$.

\begin{prop}\label{absolut-relative}
Let $M_{1}$ and $M_{2}$ be in $\mod\L$ such that $M_2$ is a generator for $\mod \L$ and
$M_1$ is a cogenerator for $\mod \L$. The following are equivalent for a positive integer
$k$:

\begin{enumerate}
\item[(a)]$M_2\perp_{k} M_{1}$,

\item[(b)]For any $C$ in $\mod\L$, $\Ext_{F_{M_2}}^{i}(C,M_1) \simeq \Ext_{\L}^{i}(C,M_1)$, for
$0<i\leq k$
\item[(c)]For any $D$ in $\mod\L$, $\Ext_{F^{M_1}}^{i}(M_2,D)\simeq \Ext_{\L}^{i}(M_2,D)$, for $0<i\leq
k$.
\end{enumerate}
\end{prop}

\begin{proof}
(a)$\Rightarrow$(b) Let $C$ be in $\mod\L$ and let
\[\cdots \to P_l\extto{f_l} \cdots \to P_1\extto{f_1} P_0 \extto{f_0} C\to 0\]
be a minimal $F_{M_2}$-projective resolution of $M_1$. Set $K_{-1}=M_1$ and $K_{i}=\Ker
f_i$, for $i=0,1,2,\ldots$. We first show that for all $i=-1,0,1,2,\ldots$ we have
\[\Ext_{F_{M_2}}^{1}(K_i,M_1)\simeq\Ext_{\L}^{1}(K_i,M_1).\]
Applying the functor $\Hom_{\L}(-,M_1)$ to the short exact sequence
\[ 0\to K_{i+1}\to P_{i+1}\to K_{i}\to 0\] we
get the long exact sequence
\[\Hom_\L(P_{i+1},M_1)\to \Hom_\L(K_{i+1},M_1)\to \Ext_{\L}^1(K_i,M_1)\to
\Ext_{\L}^1(P_{i+1},M_1).\] But since $M_2\perp_k M_1$ and $P_{i+1}$ is in $\add M_2$, we
have $\Ext_{\L}^1(P_{i+1},M_1)=(0)$. So $\Ext_{\L}^1(K_i,M_1)$ is the cokernel of the map
$\Hom_\L(K_{i+1},M_1)\to \Hom_\L(P_{i+1},M_1)$. The short exact sequence \[ 0\to
K_{i+1}\to P_{i+1}\to K_{i}\to 0\] is $F_{M_2}$-exact so we also have the following long
exact sequence
\[\Hom_\L(P_{i+1},M_1)\to \Hom_\L(K_{i+1},M_1)\to \Ext_{F_{M_2}}^1(K_i,M_1)\to
\Ext_{F_{M_2}}^1(P_{i+1},M_1)\] and since $\Ext_{F_{M_2}}^1(P_{i+1},M_1)=(0)$,
$\Ext_{F_{M_2}}^1(K_i,M_1)$ is the cokernel of the map $\Hom_\L(K_{i+1},M_1)\to
\Hom_\L(P_{i+1},M_1)$. Hence
\[\Ext_{F_{M_2}}^{1}(K_i,M_1)\simeq\Ext_{\L}^{1}(K_i,M_1),\]  for all
$i=-1,0,1,2,\ldots.$ In particular we have that
\[\Ext_{F_{M_2}}^{1}(C,M_1)\simeq\Ext_{\L}^{1}(C,M_1).\]Next we show that
for all $i=-1,0,1,2,\ldots$ and $2\leq j\leq k$
\[\Ext_{\L}^{j-1}(K_{i+1},M_1)\simeq\Ext_{\L}^{j}(K_{i},M_1).\] To do
this, we apply the functor $\Hom_{\L}(-,M_1)$ to the short exact sequence
\[ 0\to K_{i+1}\to P_{i+1}\to K_{i}\to 0\]
and we get the long exact sequences
\[\Ext_{\L}^{j-1}(P_{i+1},M_1)\to \Ext_{\L}^{j-1}(K_{i+1},M_1)\to
\Ext_{\L}^{j}(K_{i},M_1)\to \Ext_{\L}^{j}(P_{i+1},M_1).\] But since $M_2 \perp_k M_1$ and
$P_{i+1}$ is in $\add M_2$, we have that $\Ext_{\L}^{j-1}(P_{i+1},M_1)=0$ and
$\Ext_{\L}^{j}(P_{i+1},M_1)=0$ which implies that
\[\Ext_{\L}^{j-1}(K_{i+1},M_1)\simeq\Ext_{\L}^{j}(K_{i},M_1).\]

Now using the above we can see that for all $2\leq i\leq k$ we have
\[\Ext_{F_{M_2}}^{i}(C,M_1)\simeq\Ext_{F_{M_2}}^{i-1}(K_0,M_1)\simeq\cdots\simeq
\Ext_{F_{M_2}}^{1}(K_{i-2},M_1)\simeq\]\[\Ext_{\L}^{1}(K_{i-2},M_1)\simeq
\Ext_{\L}^{2}(K_{i-3},M_1)\simeq\cdots\simeq\Ext_{\L}^{i}(C,M_1)\] which
completes the proof. \\
(b)$\Rightarrow$(a) Set $C=M_2$. Then we have that $\Ext_{F_{M_2}}^{i}(M_2,M_1)=(0)$,
since $M_2$ is $F_{M_2}$-projective, hence $\Ext_{\L}^{i}(C,M_1)=(0)$, for $0<i\leq k$.

The proof of (a)$\Leftrightarrow$(c) is symmetric.
\end{proof}

\section{Cotilting and maximal orthogonal modules}
In this section we state and prove the main theorem and give the connections with Iyama's
result. But let us start by recalling Iyama's definition for maximal $l$-orthogonal
$\L$-modules. Set $M^{\perp_k}=\{Y\in\mod\L \mid M\perp_k Y\}$ and
${^{\perp_k}M}=\{X\in\mod\L \mid X\perp_k M\}$.
\begin{defin} A $\L$-module $M$ is called \emph{maximal \ $l$-orthogonal}\ if \[M^{\perp_l}
=\add M= {^{\perp_l}M}.\]
\end{defin}
The following conjecture was stated in ~\cite{I1}.

\begin{conj}[O. Iyama]
Let $M_1$ and $M_2$ be maximal $l$-orthogonal in $\mod\L$. Then their endomorphism rings,
$\End_\L(M_1)$ and $\End_\L(M_2)$, are derived equivalent.
\end{conj}
Before we continue, we give a characterization of maximal orthogonal modules that can be
 found in a more general setting in ~\cite[Proposition 2.2.2]{I1}. For convenience, we restate
 it and prove it here in the language of relative homology. We call a $\L$-module \emph{k-selforthogonal} if ${M\perp_k M}$
holds.

\begin{prop}\label{FM2 relative dimension of M1}Let $M$ be in $\mod\L$ and $l$ a positive
integer. The following are equivalent for any integer $k$, such that $0\leq k\leq l$.
\begin{enumerate}
\item[(a)] M is maximal $l$-orthogonal in $\mod\L$,
\item[(b)] M is a generator-cogenerator for $\mod\L$, $M$ is $l$-selforthogonal
and $\pd_{F_M}X\leq l-k$ for any $X$ in ${^{\perp_k}M}$,
\item[(c)] M is a generator-cogenerator for $\mod\L$, $M$ is $l$-selforthogonal
and $\id_{F^M}Y\leq l-k$ for any $Y$ in $M^{\perp_k}$.
\end{enumerate}
\end{prop}

\begin{proof}
(a)$\Rightarrow$(b) Let
\[\cdots \to M_{l-k}\extto{f_{l-k}} M_{l-k-1}\to \cdots \to M_1\extto{f_1} M_0
\extto{f_0} C\to 0\]be a minimal $F_M$-projective resolution of $X$. Set $K_{-1}=X$ and
$K_{i}=\Ker f_i$, for $i=0,1,2,\ldots$. We want to show that $K_{l-k-1}$ is in $\add M$.
In order to do this, we show that $\Ext_{\L}^{j}(M,K_{l-k-1})=(0),$ for all
$j=1,2,\ldots,l$. For all $i$, applying the functor $\Hom_{\L}(M,-)$ to the short exact
sequence \[ 0\to K_{i+1}\to M_{i+1}\to K_{i}\to 0\] we get the long exact sequence
\[\Hom_{\L}(M,M_{i+1})\to \Hom_{\L}(M,K_{i})\to
\Ext_{\L}^{1}(M,K_{i+1})\to \Ext_{\L}^{1}(M,M_{i+1}).\] Since $M$ is $l$-selforthogonal
with $l\geq 1$ and $M_{i+1}$ is in $\add M$, we have that $\Ext_{\L}^{1}(M,M_{i+1})=(0)$.
Moreover, since the short exact sequence $0\to K_{i+1}\to M_{i+1}\to K_{i}\to 0$ is
$F_{M}$-exact, the map $\Hom_{\L}(M,M_{i+1})\to \Hom_{\L}(M,K_{i})$ is an epimorphism.
Hence $\Ext_{\L}^{1}(M,K_{i+1})=(0)$. From the same short exact sequence we also get the
long exact sequences
\[\Ext_{\L}^{j-1}(M,M_{i+1})\to \Ext_{\L}^{j-1}(M,K_{i})\to
\Ext_{\L}^{j}(M,K_{i+1})\to \Ext_{\L}^{j}(M,M_{i+1}).\] For $2\leq j \leq l$, since $M$
is $l$-selforthogonal and $M_{i+1}$ is in $\add M$, we have that
$\Ext_{\L}^{j-1}(M,M_{i+1})=(0)$ and $\Ext_{\L}^{j}(M,M_{i+1})=(0)$, which implies that
\[\Ext_{\L}^{j-1}(M,K_{i})\simeq \Ext_{\L}^{j}(M,K_{i+1}), \ \ 2\leq j\leq l\]
Now, using the above, we can compute the groups
$\Ext_{\L}^{j}(M,K_{l-k-1})$ as follows:\\
for  $2\leq j\leq l-k$ we have
\[\Ext_{\L}^{j}(M,K_{l-k-1})\simeq\Ext_{\L}^{1}(M,K_{l-k-j})=(0)\]
and for $j=l-k+s, 1\leq s\leq k$ we have
\[\Ext_{\L}^{l-k+s}(M,K_{l-k-1})\simeq\Ext_{\L}^{s}(M,X)=(0).\]

So we have
\[\Ext_{\L}^{j}(M,K_{l-k-1})\simeq\Ext_{\L}^{1}(M,K_{l-k-j})=(0), \ \ 2\leq j\leq l\]
 and since $M$ is maximal $l$-orthogonal this
implies that $K_{l-k-1}$ is in $\add M$. Hence $\pd_{F_{M}}X\leq l-k$.

(b)$\Rightarrow$(a) Let $X$ be in $\mod\L$ such that $\Ext_\L^i(X,M)=(0)$, for $0<i\leq
l$. We will show that $X$ is then in $\add M$. Note that $k\leq l$, so
$\Ext_\L^i(X,M)=(0)$, for $0<i\leq k$ or equivalently $X$ is in ${^{\perp_k}M}$, hence by
assumption $\pd_{F_M}X\leq l-k$. If $k=0$, this will imply that $X$ is $F_M$-projective,
 hence $X$ is in $\add M$. Assume that $k>0$ and let
\[0\to M_{l-k}\extto{f_{l-k}} M_{l-k-1}\to\cdots\to M_1\extto{f_1}M_0\extto{f_0} X\to 0\]
be a minimal $F_M$-projective resolution of $X$. Set $K_{-1}=X$, $K_{i}=\Ker f_i$, for
$i=0,1,\ldots,l-k-2$ and $K_{l-k-1}=M_{l-k}$. Then, for any $i$, applying the functor
$\Hom_\L(X,-)$ to the short exact sequence
\[0\to K_i\to M_i\to K_{i-1}\to 0\]
we get long exact sequences
\[\Ext_\L^j(X,M_i)\to\Ext_\L^j(X,K_{i-1})\to\Ext_\L^{j+1}(X,K_i)\to\Ext_\L^{j+1}(X,M_i).\]
Since $M$ is $l$-selforthogonal and $M_i$ is in $\add M$ for all $i$, we have that the
first and last term of the above sequence vanish for all $j=1,2,\ldots,l-1$. So we have
\[\Ext_\L^j(X,K_{i-1})\simeq\Ext_\L^{j+1}(X,K_i)\]
for all $i$ and for $j=1,2,\ldots,l-1$. But then we have
\[\Ext_\L^1(X,K_0)\simeq\Ext_\L^2(X,K_1)\simeq\cdots\simeq\Ext_\L^{l-k}(X,K_{l-k-1})=
\Ext_\L^{l-k}(X,M_{l-k})=(0).\] This implies that the short exact sequence $0\to K_0\to
M_0\to X\to 0$ splits and hence $X$ is in $\add M$. So $M$ is maximal $l$-orthogonal.

The proof of (a)$\Leftrightarrow$(c) is symmetric.
\end{proof}
Setting $k=0$ in the above proposition, and using Lemma \ref{gldim}, we have the
following nice characterization of maximal $l$-orthogonal modules.
\begin{cor}\label{maximal orthogonal}

Let $M$ be in $\mod\L$ and $l$ a positive integer. The following are equivalent:
\begin{enumerate}
\item[(a)] $M$ is maximal $l$-orthogonal,
\item[(b)] $M$ is a generator-cogenerator for $\mod\L$, $M$ is $l$-selforthogonal \\and
$\gldim\End_\L(M)\leq l+2$.
\end{enumerate}
\end{cor}

Before we state and prove our main theorem we recall some definitions. A $\L$-module $M$
is called \emph{cotilting}, if it has the following properties: (1) $\Ext_\L^i(M,M)=(0)$,
for $i>0$, (2) $\id_\L M<\infty$ and (3) $\I(\L)\subseteq \widehat{\add M}$. Similarly,
if $F=F_M$ is a sub-bifunctor of $\Ext_\L^1(-,-)$, a $\L$-module $M$ is called
\emph{$F$-cotilting} if : (1) $\Ext_F^i(M,M)=(0)$, for $i>0$, (2) $\id_F M < \infty$ and
(3) $\I(F)\subseteq \widehat{\add_F M}$, where $\widehat{\add_F M}$ denotes the full
subcategory of $\mod\L$ consisting of all modules that have a finite resolution in $\add
M$ which is in addition $F$-exact. The notions of \emph{tilting} and \emph{F-tilting}
modules are defined dually.

For an artin algebra $\L$ it is known that when $\gldim \L<\infty$, a $\L$-module $M$ is
tilting if and only if $M$ it is cotilting. The proof is based on the one to one
correspondences between equivalence classes of tilting or cotilting $\L$-modules and
certain subcategories of $\mod\L$ (see for example ~\cite[Theorem 2.1]{S}). A relative
version of these correspondences is proved in ~\cite{AS2} and using these one can prove
the following:
\begin{lem}\label{tilt-cotilt}
Let $F$ be a sub-bifunctor of $\Ext^1_\L(-,-)$ with enough projectives and  $M$ a
$\L$-module. Assume that $\gldim_{F}\L<\infty$. Then $M$ is  $F$-tilting  if and only if
$M$ is $F$-cotilting.
\end{lem}
We are now in position to give the main theorem of this paper.
\begin{thm}\label{iyama}
Let $\L$ be an artin algebra with $M_1$ and $M_2$ two generator-cogenerators for
$\mod\L$. Suppose that there exists some positive integer $l$ such that
$\gldim\End_\L(M_{i})\leq l+2$ for $i=1,2$. The following are equivalent:
\begin{enumerate}
\item[(a)]
$\Ext^i_{F^{M_1}}(M_2,M_2)=(0)$ and $\Ext^i_{F_{M_2}}(M_1,M_1)=(0)$ for $0<i\leq l$,
\item[(b)] $M_2$ is an $F^{M_1}$-cotilting module,
\item[(c)] $M_1$ is an $F_{M_2}$-cotilting module,
\item[(d)] $\Hom_\L(M_2,M_1)$ is a cotilting
$\End_\L(M_2)^\op$-$\End_\L(M_1)$-bimodule.
\end{enumerate}
\end{thm}

\begin{proof}
(a)$\Rightarrow$(b) First, since $\gldim \End_\L(M_{1})\leq l+2$, applying Lemma
\ref{gldim}, we see that $\id_{F^{M_1}} M_2\leq l$. Also, by assumption, we have
$\Ext^i_{F^{M_1}}(M_2,M_2)=(0)$ for $0<i\leq l$ and since $\id_{F^{M_1}} M_2\leq l$ we
get $\Ext^i_{F^{M_1}}(M_2,M_2)=(0)$ for $i>0$. It remains to show that $\I(F^{M_1})$ is
contained in $\widehat{\add_{F^{M_1}} M_2}$. To do this, we consider a minimal
$F_{M_2}$-projective resolution of $M_1$
\[\cdots\to P_l\to P_{l-1}\extto{f_{l-1}} \cdots\to P_1\to P_0\to
M_1\to 0\] and we set $K_{l-1}=\Ker f_{l-1}$. Since $\gldim \End_\L(M_{2})\leq l+2$,
applying Lemma \ref{gldim} we see that $K_{l-1}$ is $F_{M_2}$-projective, so it is in
$\add M_2$. Applying the functor $\Hom_\L(-,M_1)$ to this sequence we get a long exact
sequence
\begin{multline}
0\to \Hom_\L(M_1,M_1)\to \Hom_\L(P_0,M_1)\to \Hom_\L(P_1,M_1)\to\cdots\to \notag \\
\Hom_\L(P_{l-1},M_1)\to \Hom_\L(K_{l-1},M_1).
\end{multline}
But, by assumption, $\Ext^i_{F_{M_2}}(M_1,M_1)=(0)$ for $0<i\leq l$, so the above
sequence is exact and the map $\Hom_\L(P_{l-1},M_1)\to \Hom_\L(K_{l-1},M_1)$ is an
epimorphism, which implies that the sequence
\[0\to K_{l-1}\to P_{l-1}\extto{f_{l-1}}\cdots\to P_1\to P_0\to
M_1\to 0\] is $F^{M_1}$-exact. Thus $M_2$ is an $F^{M_1}$-cotilting
module.\\

(b)$\Rightarrow$(a) Since $M_2$ is an $F^{M_1}$-cotilting module,
$\Ext^i_{F^{M_1}}(M_2,M_2)=(0)$ for $i>0$. Moreover, $\I(F^{M_1})$ is contained in
$\widehat{\add_{F^{M_1}} M_2}$, so there exists an $F^{M_1}$-exact sequence
\[(\eta)\colon 0\to P_n\to P_{n-1}\extto{f_{n-1}}\cdots\to P_1\extto{f_1} P_0\extto{f_0} M_1\to 0\]
with $P_i$ in $\add M_2$, for all $i$. We will show that $(\eta)$ is also
$F_{M_2}$-exact. We set $K_i=\Ker f_i$ for $i=1,\ldots,n-2$, $K_0=M_1$ and $K_{n-1}=P_n$.
Then, for any $i$, applying the functor $\Hom_\L(M_2,-)$ to the short exact sequence
\[(\eta_i)\colon 0\to K_i \to P_i \to K_{i-1}\] we get long exact sequences
\[\Ext_{F^{M_1}}^j(M_2,P_i)\to \Ext_{F^{M_1}}^j(M_2,K_{i-1})\to
\Ext_{F^{M_1}}^{j+1}(M_2,K_i)\to \Ext_{F^{M_1}}^{j+1}(M_2,P_i)\] for $j>0$. But since
$P_i$ is in $\add M_2$, for all i, the first and the last term of these sequences are
zero, hence the two middle terms are isomorphic. So we have
\begin{multline}
\Ext_{F^{M_1}}^1(M_2,K_i)\simeq\Ext_{F^{M_1}}^2(M_2,K_{i+1})\simeq
\cdots\simeq\Ext_{F^{M_1}}^{n-i}(M_2,K_{n-1})=\notag \\
\Ext_{F^{M_1}}^{n-i}(M_2,P_n)=(0).
\end{multline}
for $i=0,1,\ldots,n-1$. But by applying the functor $\Hom_\L(M_2,-)$ to the short exact
sequences $(\eta_i)$, we also get the long exact sequences
\[\Hom_\L(M_2,P_i)\to\Hom_\L(M_2,K_{i-1})\to\Ext_{F^{M_1}}^1(M_2,K_i)\]
and since $\Ext_{F^{M_1}}^1(M_2,K_i)=(0)$ for all $i$, we have that the map
$\Hom_\L(M_2,P_i)\to\Hom_\L(M_2,K_{i-1})$ is an epimorphism for all $i$, which implies
that $(\eta)$ is $F_{M_2}$-exact. We now apply the functor $\Hom_\L(-,M_1)$ to $(\eta)$
and we get the complex
\[0\to\Hom_\L(M_1,M_1)\to\Hom_\L(P_0,M_1)\to\Hom_\L(P_1,M_1)\to\cdots\to\Hom_\L(P_n,M_1)\to 0\]
Since $(\eta)$ is $F_{M_2}$-exact, the $i-th$-homology of the above complex is
$\Ext_{F_{M_2}}^i(M_1,M_1)$. But since $(\eta)$ is $F^{M_1}$-exact, the above complex is
acyclic. Hence $\Ext_{F_{M_2}}^i(M_1,M_1)=(0)$, for $i>0$, which completes the proof.

(a)$\Leftrightarrow$(c)
The proof is symmetric to the proof of (a)$\Leftrightarrow$(b)\\

(b)$\Rightarrow$(d) Since $\gldim \End_\L(M_{1})\leq l+2$, by Lemma \ref{gldim}, we get
$\gldim_{F^{M_1}}\L\leq l$ and then, by Lemma \ref{tilt-cotilt}, we have that $M_2$ is an
$F^{M_1}$-tilting $\L$-module. But then, since $\I(F^{M_1})=\add M_1$, the module
$\Hom_\L(M_2,M_1)$ is a cotilting $\End_\L(M_2)^\op$-module, as shown in \cite{AS2}. This
is equivalent to $\Hom_\L(M_2,M_1)$ being a cotilting
$\End_\L(M_2)^\op$-$\End_\L(M_1)$-bimodule.\\

(d)$\Rightarrow$(a) We first show that $\Ext^i_{F^{M_1}}(M_2,M_2)=(0)$ for $0<i\leq l$.
Recall that by Lemma \ref{gldim} we have that $\gldim_{F^{M_1}}\L\leq l$ and let
\[0\to M_2\to I_0\to I_1\to\cdots\to I_l\to 0\] be a minimal
$F^{M_1}$-injective resolution of $M_2$. Then by applying the functor \\$\Hom_\L(-,M_1)$
we get a minimal projective resolution of the $\End_\L(M_1)$-module $\Hom_\L(M_2,M_1)$
\begin{multline}
0\to \Hom_\L(I_l,M_1)\to \Hom_\L(I_{l-1},M_1)\to \cdots\to\Hom_\L(I_1,M_1)\to \notag \\
\Hom_\L(I_0,M_1)\to \Hom_\L(M_2,M_1)\to 0
\end{multline} Applying the functor $\Hom_{\End_\L(M_1)}(-,\Hom_\L(M_2,M_1))$
to the last sequence we get the following commutative exact diagram where the notation
has been simplified
\[\xymatrix@C=7pt{
0\ar[r]&((M_2,M_1),(M_2,M_1))\ar[d]^-{\wr}\ar[r]&((I_0,M_1),(M_2,M_1))\ar[d]^-{\wr}\ar[r]&
\cdots\ar[r]&((I_l,M_1),(M_2,M_1))\ar[d]^-{\wr}\ar[r]&0\\
0\ar[r]&(M_2,M_2)\ar[r]&(M_2,I_0)\ar[r]&\cdots\ar[r]&(M_2,I_l)\ar[r]&0}\] From the above
diagram we see that
\[\Ext^i_{F^{M_1}}(M_2,M_2)\simeq
\Ext^i_{\End_\L(M_1)}( \Hom_\L(M_2,M_1),\Hom_\L(M_2,M_1) )=(0),\]
for $i>0$, since $\Hom_\L(M_2,M_1)$ is a cotilting $\End_\L(M_1)$-module.\\
Symmetrically, starting with a minimal $F_{M_2}$-projective resolution of $M_1$ and
applying the functor $\Hom_\L(M_2,-)$ we can show that
\[\Ext^i_{F_{M_2}}(M_1,M_1)\simeq
\Ext^i_{\End_\L(M_2)^{\op}}(\Hom_\L(M_2,M_1),\Hom_\L(M_2,M_1)) =(0)\] for $i>0$. Here we
are using that $\Hom_\L(M_2,M_1)$ is a cotilting $\End_\L(M_2)^{\op}$-module.
\end{proof}
The following easy consequence of the above theorem generalizes the Theorem 5.3.2 in
~\cite{I1}.
\begin{cor}
Let $M_1$ and $M_2$ be maximal $l$-orthogonal modules in $\mod\L$ such that
$\Ext^i_{F^{M_1}}(M_2,M_2)=(0)$ and $\Ext^i_{F_{M_2}}(M_1,M_1)=(0)$ for $0<i\leq l$. Then
their endomorphism rings, $\End_\L(M_1)$ and $\End_\L(M_2)$, are derived equivalent.
\end{cor}

\begin{proof}
Let $M_1$ and $M_2$ be maximal $l$-orthogonal modules in $\mod\L$ satisfying the
assumption of the Corollary. By Proposition \ref{maximal orthogonal} and Proposition
\ref{gldim} we have that $M_1$ and $M_2$ are generator-cogenerators for $\mod\L$ with
$\gldim\End_\L(M_i)\leq l+2$, $i=1,2$. Then using Theorem \ref{iyama} and we get that
$\Hom_\L(M_2,M_1)$ is a cotilting $\End_\L(M_2)^\op$-$\End_\L(M_1)$-bimodule and hence
$\End_\L(M_1)$ and $\End_\L(M_2)$ are derived equivalent by a result of Happel ~\cite{H}.
\end{proof}

Although it is obvious from our Theorem \ref{iyama} that Iyama's orthogonality condition
on two maximal $l$-orthogonal modules $M_1$ and $M_2$  implies the vanishing of
$\Ext^i_{F^{M_1}}(M_2,M_2)$ and $\Ext^i_{F_{M_2}}(M_1,M_1)$ for $0<i\leq l$, it is
interesting to give a direct proof of this fact. This is done in the next proposition.

\begin{prop}\label{iyama-us}
Let $M_1$ and $M_2$ be maximal $l$-orthogonal in $\mod \L$. Assume that there exists a
positive integer $k$, such that $k\leq l\leq 2k+1$ and $M_2\perp_k M_1$. Then
\begin{enumerate}
\item[(a)]$\Ext_{F_{M_2}}^{i}(M_1,M_1)=(0)$, $0<i\leq l$,
\item[(b)]$\Ext_{F^{M_1}}^{i}(M_2,M_2)=(0)$, $0<i\leq l$.
\end{enumerate}
\end{prop}

\begin{proof}
(a) For $0<i\leq k$, since $M_2\perp_k M_1$, by Proposition \ref{absolut-relative} we
have that
\[\Ext_{F_{M_2}}^{i}(M_1,M_1)=\Ext_{\L}^{i}(M_1,M_1).\] But $M_1$ is
$l$-orthogonal, so $\Ext_\L^{i}(M_1,M_1)=(0)$, for $0<i\leq k$. Hence
$\Ext_{F_{M_2}}^{i}(M_1,M_1)=(0)$, for $0<i\leq k$.

For $i> k+1$, since $l\leq 2k+1$, we have that $i> l-k$. But since $M_2$ is maximal
$l$-orthogonal, by Lemma \ref{FM2 relative dimension of M1} we have that
$\pd_{F_{M_2}}M_1\leq {l-k}$. So $\Ext_{F_{M_2}}^{i}(M_1,M_1)=(0)$, for $i> k+1$.

For $i=k+1$, if $l<2k+1$, we have again that $i> l-k$, so using the same argument as
before we get that $\Ext_{F_{M_2}}^{i}(M_1,M_1)=(0)$. It remains to show that
$\Ext_{F_{M_2}}^{k+1}(M_1,M_1)=(0)$, in the case where $l=2k+1$. In order to do this,
consider a minimal $F_{M_2}$-projective resolution of $M_1$
\[\cdots\to P_k\extto{f_k} P_{k-1}\to \cdots\to P_1\extto{f_1} P_0\to M_1\to 0\]
and set $K_i=\Ker f_i$, for $i=0,1,\ldots$. Then we have
\begin{multline}
\Ext_{F_{M_2}}^{k+1}(M_1,M_1)\simeq\Ext_{F_{M_2}}^{k}(K_0,M_1)\simeq\cdots\simeq
\Ext_{F_{M_2}}^{1}(K_{k-1},M_1)\simeq \notag
\\ \Ext_{\L}^{1}(K_{k-1},M_1)\simeq\cdots\simeq\Ext_{\L} ^{k}(K_0,M_1).
\end{multline}
To show these isomorphisms we use the same arguments as in the proof of
\ref{absolut-relative} and we omit here the details. Applying the functor
$\Hom_{\L}(-,M_1)$ to the short exact sequence
\[
0\to K_0 \to P_0 \to M_1\to 0\] we get the long exact sequence
\[\Ext_{\L}^{k}(P_0,M_1)\to\Ext_{\L}^{k}(K_0,M_1)\to
\Ext_{\L}^{k+1}(M_1,M_1).\] Since $M_1$ is $l$-orthogonal and $k+1<l$, we have that
$\Ext_{\L}^{k+1}(M_1,M_1)=(0)$ and since $M_2\perp_k M_1$ and $P_0$ is in $\add M_2$, we
have that $\Ext_{\L}^{k}(P_0,M_1)=(0)$. So $\Ext_{\L}^{k}(K_0,M_1)=(0)$ and hence
$\Ext_{F_{M_2}}^{k+1}(M_1,M_1)=(0)$, which completes the proof.

The proof of (b) is symmetric.
\end{proof}

The converse of Proposition \ref{iyama-us} is not in general true, as the following
example shows.

\begin{example}

Let $Q$ be the quiver

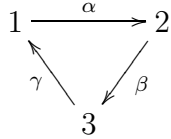
\[\xymatrix@C=15pt{1\ar[rr]^{\alpha}&&2\ar[dl]^{\beta}\\
&3\ar[ul]^{\gamma}& }\] and $KQ$ the path algebra of $Q$ over some field $K$. Let also
$I$ be the ideal of $KQ$ generated by all paths of length $5$ and consider the factor
algebra $\L=KQ/I$. Set
\[M_1=P_1\oplus P_2\oplus P_3\oplus S_1\oplus P_3/\rrad^2 P_3\]
and
\[M_2=P_1\oplus P_2\oplus P_3\oplus S_1\oplus P_1/\rrad^2 P_1,\]
where $P_i$ denotes the indecomposable projective $\L$-module corresponding to vertex
$i$, for $i=1,2,3$ and  $S_1$ denotes the simple module in vertex $1$. It is easy to
verify that $M_1$ and $M_2$ are maximal 2-orthogonal modules in $\mod\L$. Moreover, the
modules $M_1$ and $M_2$ are connected via the following exact sequence:

\[(\eta)\colon 0\to P_1/\rrad^2 P_1\to S_1\oplus P_1\to S_1\oplus P_3\to P_3/\rrad^2 P_3\to 0.\]
Observe that the above sequence is both an $F_{M_2}$-exact and an $F^{M_1}$-exact
sequence. Hence $(\eta)$ can be viewed both as an $F_{M_2}$-projective resolution of
$P_3/\rrad^2 P_3$ and as an $F^{M_1}$-injective resolution of $P_1/\rrad^2 P_1$. But
then, if we apply the functor $\Hom_\L(-,M_1)$ to $(\eta)$, the resulting complex is
acyclic, since $(\eta)$ is $F^{M_1}$-exact, so
$\Ext_{F_{M_2}}^i(P_3/\rrad^2P_3,M_1)=(0)$, for $i>0$ and hence
$\Ext_{F_{M_2}}^i(M_1,M_1)=(0)$, for $i>0$. Also, if we apply the functor
$\Hom_\L(M_2,-)$ to $(\eta)$, the resulting complex is acyclic, since $(\eta)$ is
$F_{M_2}$-exact, so $\Ext_{F^{M_1}}^i(M_2,P_1/\rrad^2 P_1)=(0)$, for $i>0$ and hence
$\Ext_{F^{M_1}}^i(M_2,M_2)=(0)$, for $i>0$. Thus, we have shown that $M_1$ and $M_2$
satisfy the conclusions (a) and (b) of Proposition \ref{iyama-us}. Next we consider the
short exact sequence
\[0\to P_3/\rrad^2 P_3 \to P_1/ \Soc P_1 \to P_1/\rrad^2 P_1\to 0.\]
The sequence is non split, so $\Ext_\L^1(P_1/\rrad^2 P_1,P_3/\rrad^2 P_3)\neq(0)$, hence
$\Ext_\L^1(M_2,M_1)\neq(0)$ which implies that Iyama's orthogonality condition does not
hold for $M_1$ and $M_2$.
\end{example}
In fact, in the above example, we can compute all maximal $2$-orthogonal $\L$-modules and
we can then see that they are all connected with sequences which have the properties of
$(\eta)$, meaning that we can get one from the other by exchanging one indecomposable
summand using approximations. Thus, for the above example, Iyama's conjecture is true. We
complete this section with a proposition showing the connection between these exchange
sequences and the vanishing of the relative $\Ext_F$.
\begin{prop}
Let $M_i=N\oplus X_i$ be in $\mod\L$ such that $N$ is a generator-cogenerator for
$\mod\L$ and $X_i$ are indecomposable not contained in $\add N$, for $i=1,2$. Assume that
there exists a positive integer $l$ such that $\gldim\End_\L(M_i)\leq l+2$, for $i=1,2$.
Then the following are equivalent:
\begin{enumerate}
\item[(a)]$\Ext^i_{F^{M_1}}(M_2,M_2)=(0)$ and $\Ext^i_{F_{M_2}}(M_1,M_1)=(0)$ for $0<i\leq
l$,
\item[(b)] there exists an exact sequence
\[(\eta)\colon 0\to X_2\to N_0\extto{f_0} N_1\to\cdots\to N_m\extto{f_m} X_1\to 0\]
where each map $\Ker f_j\to M_j$ is a minimal left $\add N$-approximation, each map
$M_j\to \Im f_j$ is a minimal right $\add N$-approximation and $(\eta)$ is in addition
$F^{M_1}$-exact and $F_{M_2}$-exact.
\end{enumerate}
\end{prop}

\begin{proof}
(a)$\Rightarrow$(b) Using Theorem \ref{iyama} we have that $\Hom_\L(M_2,M_1)$ is a
cotilting $\End_\L(M_1)$-module. Also we know that $\Hom_\L(M_1,M_1)$ is trivially a
cotilting $\End_\L(M_1)$-module. So the $\End_\L(M_1)$-modules $\Hom_\L(X_2,M_1)$ and
$\Hom_\L(X_1,M_1)$ are complements of the almost complete cotilting $\End_\L(M_1)$-module
$\Hom_\L(N,M_1)$. Then, by ~\cite{CHU}, there exists an exact sequence
\begin{multline}0\to\Hom_\L(X_1,M_1)\extto{f_m^\star}\Hom_\L(N_m,M_1)\to\cdots\extto{f_1^\star}
\Hom_\L(N_1,M_1)\extto{f_0^\star}\notag\\
\Hom_\L(N_0,M_1)\to\Hom_\L(X_2,M_1)\to 0\end{multline} where each map $\Im
f_j^\star\to\Hom_\L(N_j,M_1)$ is a minimal left $\add \Hom_\L(N,M_1)$-approximation and
each $\Hom_\L(N_j,M_1)\to \Coker f_j^\star$ is a minimal right $\add
\Hom_\L(N,M_1)$-approximation. Since $M_1$ is a cogenerator for $\mod\L$, we have that
for all $j$, $f_j^\star=\Hom_\L(f_j,M_1)$ for some $f_j\colon M_j\to M_{j+1}$ and the
sequence
\[(\eta)\colon 0\to X_2\to N_0 \extto{f_0} N_1\extto{f_1}\cdots\to N_m\extto{f_m} X_1\to 0\]
is $F^{M_1}$-exact. Moreover each map $\Ker f_j\to M_j$ is a minimal left $\add
N$-approximation, each map $M_j\to \Im f_j$ is a minimal right $\add N$-approximation. It
remains to show that $(\eta)$ is in also $F_{M_2}$-exact. To do this, we apply the
functor $\Hom_\L(M_2,-)$ to $(\eta)$ and we get the complex
\begin{multline}0\to \Hom_\L(M_2,X_2)\to\Hom_\L(M_2,N_0)\to\Hom_\L(M_2,N_1)\to
\cdots\to\notag \\
\Hom_\L(M_2,N_m)\to\Hom_\L(M_2,X_1)\to 0.\end{multline}

But $(\eta)$ can be viewed as an $F^{M_1}$-injective resolution of $X_2$ and then the
$j$-th-homology of the above complex is $\Ext_{F^{M_1}}^j(M_2,X_2)$ which is, by
assumption, zero. So the complex is acyclic which implies that $(\eta)$ is
$F_{M_2}$-exact.

(b)$\Rightarrow$(a) Assume that there exists a sequence $(\eta)$ as in (b). Then $(\eta)$
can be viewed both as an $F^{M_1}$-injective resolution of $X_2$ and as an
$F_{M_2}$-projective resolution of $X_1$. If we apply the functor $\Hom_\L(M_2,-)$ to
$(\eta)$, the resulting complex will be acyclic since $(\eta)$ is $F_{M_2}$-exact, hence
$\Ext_{F^{M_1}}^i(M_2,X_2)=(0)$, for $i>0$, which implies that $\Ext_\L^i(M_2,M_2)=(0)$,
for $i>0$, since $N$ is $F^{M_1}$-injective. Similarly, if we apply the functor
$\Hom_\L(-,M_1)$ to $(\eta)$, the resulting complex will be acyclic since $(\eta)$ is
$F^{M_1}$-exact, hence $\Ext_{F_{M_2}}^i(X_1,M_1)=(0)$, for $i>0$, which implies that
$\Ext_\L^i(M_2,M_2)=(0)$, for $i>0$, since $N$ is $F_{M_2}$-injective.
\end{proof}
Note that the integer $m$ that appears in the sequence $(\eta)$ of the above proposition
can be at most $l-1$, by Proposition \ref{gldim}.

\end{document}